\documentclass[12pt]{article}

\usepackage{graphicx}
\usepackage{amssymb}
\usepackage{amsthm}
\usepackage{amsmath}
\usepackage{mathtools}
\usepackage{tikz}
\usepackage{pgfplots}
\pgfplotsset{compat=1.18}
\usepackage[hidelinks]{hyperref}
\usepackage{stmaryrd}
\usepackage{fullpage}
\usepackage{ifthen}
\usepackage{cite}
\usepackage{rotating}

\usepackage{fontawesome5}
\usepackage{subcaption}

\tikzset{big_vertex/.style={inner sep=8pt, outer sep=0pt, circle, fill=gray}} 
\tikzset{big_vertex_center/.style={inner sep=6pt, outer sep=0pt, circle, fill=gray}} 
\tikzset{big_turan_vertex/.style={inner sep=8pt, outer sep=0pt, circle, draw=black, fill=white}} 
\tikzset{big_turan_vertex_center/.style={inner sep=6pt, outer sep=0pt, circle, draw=black, fill=white}}

\newcommand{\vc}[1]{\ensuremath{\vcenter{\hbox{#1}}}}

\tikzset{flag_pic/.style={scale=1}}  
\tikzset{unlabeled_vertex/.style={inner sep=1.7pt, outer sep=0pt, circle, fill}}
\tikzset{labeled_vertex/.style={inner sep=3pt, outer sep=0pt, rectangle, fill=white, draw=black}}
\tikzset{edge_color0/.style={color=black,line width=1.2pt,opacity=0.5,dashed}}
\tikzset{edge_color1/.style={color=red,  line width=1.2pt,opacity=0}} 
\tikzset{edge_color2/.style={color=black, line width=1.2pt,opacity=1}}
\tikzset{edge_color3/.style={color=green!80!black,line width=1.2pt}}
\tikzset{edge_color4/.style={color=orange, line width=1.2pt}}
\tikzset{edge_color5/.style={color=red,  line width=1.2pt,dotted}}
\tikzset{edge_color6/.style={color=blue, line width=1.2pt,dotted}}
\tikzset{edge_color7/.style={color=green, line width=1.2pt,dotted}}
\tikzset{edge_color8/.style={color=gray, line width=1.2pt}}
\tikzset{edge_color9/.style={color=gray, dotted, line width=1.2pt}}
\tikzset{edge_color10/.style={color=gray, dashed, line width=1.2pt}}
\tikzset{edge_color11/.style={color=pink, dashed, line width=1.2pt}}
\tikzset{edge_f/.style={-latex}}
\tikzset{edge_b/.style={latex-}}
\tikzset{edge_color1f/.style={edge_f,edge_color1}}
\tikzset{edge_color1b/.style={edge_b,edge_color1}}
\tikzset{edge_color2f/.style={edge_f,edge_color2}}
\tikzset{edge_color2b/.style={edge_b,edge_color2}}
\tikzset{edge_color3f/.style={edge_f,edge_color3}}
\tikzset{edge_color3b/.style={edge_b,edge_color3}}
\tikzset{edge_color4f/.style={edge_f,edge_color4}}
\tikzset{edge_color4b/.style={edge_b,edge_color4}}
\tikzset{edge_color5f/.style={edge_f,edge_color5}}
\tikzset{edge_color5b/.style={edge_b,edge_color5}}
\tikzset{edge_color6f/.style={edge_f,edge_color6}}
\tikzset{edge_color6b/.style={edge_b,edge_color6}}
\tikzset{edge_color7f/.style={edge_f,edge_color7}}
\tikzset{edge_color7b/.style={edge_b,edge_color7}}
\tikzset{edge_color8f/.style={edge_f,edge_color8}}
\tikzset{edge_color8b/.style={edge_b,edge_color8}}
\tikzset{edge_colorroot/.style={color=red, line width=1.7pt}}
\tikzset{edge_thin/.style={color=black}}
\tikzset{edge_hidden/.style={color=black,dotted,opacity=0}}
\tikzset{vertex/.style={inner sep=1.7pt, outer sep=0pt, circle}}
\tikzset{vertex_color0/.style={inner sep=1.7pt, outer sep=0pt, draw, circle, fill=white}}
\tikzset{vertex_color1/.style={inner sep=1.7pt, outer sep=0pt, draw, circle, fill=red!30!white}}
\tikzset{vertex_color2/.style={inner sep=1.7pt, outer sep=0pt, draw, circle, fill=blue!30!white}}
\tikzset{vertex_color3/.style={inner sep=1.7pt, outer sep=0pt, draw, circle, fill=green}}
\tikzset{vertex_color4/.style={inner sep=1.7pt, outer sep=0pt, draw, circle, fill=yellow}}
\tikzset{vertex_color5/.style={inner sep=1.7pt, outer sep=0pt, draw, circle, fill=gray!30!white}}
\tikzset{vertex_color6/.style={inner sep=1.7pt, outer sep=0pt, draw, circle, fill=gray,label=below:{$6$}}}
\tikzset{vertex_color7/.style={inner sep=1.7pt, outer sep=0pt, draw, circle, fill=gray,label=below:{$7$}}}
\tikzset{vertex_color8/.style={inner sep=1.7pt, outer sep=0pt, draw, circle, fill=gray,label=below:{$8$}}}
\tikzset{vertex_color9/.style={inner sep=1.7pt, outer sep=0pt, draw, circle, fill=gray,label=below:{$9$}}}
\tikzset{vertex_color10/.style={inner sep=1.7pt, outer sep=0pt, draw, circle, fill=gray,label=below:{$10$}}}
\tikzset{vertex_color11/.style={inner sep=1.7pt, outer sep=0pt, draw, circle, fill=gray,label=below:{$11$}}}
\tikzset{vertex_color12/.style={inner sep=1.7pt, outer sep=0pt, draw, circle, fill=gray,label=below:{$12$}}}
\tikzset{vertex_color13/.style={inner sep=1.7pt, outer sep=0pt, draw, circle, fill=gray,label=below:{$13$}}}
\tikzset{vertex_color14/.style={inner sep=1.7pt, outer sep=0pt, draw, circle, fill=gray,label=below:{$14$}}}
\tikzset{labeled_vertex_color0/.style={inner sep=3pt, outer sep=0pt, draw, rectangle, fill=white}}
\tikzset{labeled_vertex_color1/.style={inner sep=3pt, outer sep=0pt, draw, rectangle, fill=red!30!white}}
\tikzset{labeled_vertex_color2/.style={inner sep=3pt, outer sep=0pt, draw, rectangle, fill=blue!30!white}}
\tikzset{labeled_vertex_color3/.style={inner sep=3pt, outer sep=0pt, draw, rectangle, fill=green}}
\tikzset{labeled_vertex_color4/.style={inner sep=3pt, outer sep=0pt, draw, rectangle, fill=yellow}}
\tikzset{labeled_vertex_color5/.style={inner sep=3pt, outer sep=0pt, draw, rectangle, fill=gray!30!white}}
\tikzset{labeled_vertex_color6/.style={inner sep=3pt, outer sep=0pt, draw, rectangle, fill=gray,label=below:{$6$}}}
\tikzset{labeled_vertex_color7/.style={inner sep=3pt, outer sep=0pt, draw, rectangle, fill=gray,label=below:{$7$}}}
\tikzset{labeled_vertex_color8/.style={inner sep=3pt, outer sep=0pt, draw, rectangle, fill=gray,label=below:{$8$}}}
\tikzset{labeled_vertex_color9/.style={inner sep=3pt, outer sep=0pt, draw, rectangle, fill=gray,label=below:{$9$}}}
\tikzset{text_color0/.style={color=black}}
\tikzset{text_color1/.style={color=red}}
\tikzset{text_color2/.style={color=blue}}
\tikzset{text_color3/.style={color=green!70!black}}
\tikzset{text_color4/.style={color=orange}}
\tikzset{text_color5/.style={color=gray}}

\def\outercycle#1#2{
\pgfmathtruncatemacro{\plusone}{#1+1}
\pgfmathtruncatemacro{\zeroshift}{270 - (#2-1)*360/#1/2 }
\draw  \foreach \x in {0,1,...,#1}{(\zeroshift+\x*360/#1:1) node[vertex](x\x){}};
}
    
\def\labelvertex#1{\pgfmathtruncatemacro{\vertexlabel}{#1+1 } \draw (x#1) node{\color{black}\tiny\vertexlabel}; }

\usepackage{listofitems}
\tikzset{vertex_u/.style={unlabeled_vertex}}
\tikzset{vertex_l/.style={labeled_vertex}}
\newcommand{\Fnv}[2]{ 
\ifnum#2<#1 \draw (x#2) node[vertex_l]{}; \labelvertex{#2}  
\else  \draw (x#2) node[vertex_u]{}; \fi 
}
\newcommand{\Fne}[3]{
\draw[edge_color#3] (x#1)--(x#2); 
}
\newcounter{Fneid} 
\newcommand{\Fe}[3]{
\ifnum#1=1
  \vc{\begin{tikzpicture}[scale=0.4]\outercycle{1}{2}
  \Fnv{#2}{0}
  \end{tikzpicture}}
\else
\setsepchar{ }
\readlist\elabel{#3}
\pgfmathtruncatemacro{\vertexloop}{#1-1}
\pgfmathtruncatemacro{\vertexloopi}{#1-2}
\pgfmathtruncatemacro{\expectededges}{#1*(#1-1)/2}
\ifnum\elabellen=\expectededges%
\def\cycleshift{2}
\ifnum#1=2\def\cycleshift{1}\fi
\ifnum#1=3\ifnum#2=1\def\cycleshift{1}\fi\fi
\def\Fnscale{0.4}\ifnum#1>4\def\Fnscale{0.45}\fi\ifnum#1>5\def\Fnscale{0.55}\fi\ifnum#1>7\def\Fnscale{0.65}\fi
\vc{\begin{tikzpicture}[scale=\Fnscale]
          \outercycle{#1}{\cycleshift}
          \setcounter{Fneid}{1}         
          \foreach\i in {0,...,\vertexloopi}{   
          \pgfmathtruncatemacro{\jfrom}{\i+1}
          \foreach\j in {\jfrom,...,\vertexloop}{
          \edef\eID{\arabic{Fneid}}
          \edef\eij{\elabel[\eID]}         
            \Fne\i\j{\eij}              
	    \stepcounter{Fneid}
          }}
          \foreach\i in {0,...,\vertexloop}{\Fnv{#2}{\i}  
          }
          \end{tikzpicture}}%
\else
   #1 vertices need \expectededges{} edges but got \elabellen edges.
\fi
\fi
}

\newcommand{\Fuu}[1]{
\,\vc{\begin{tikzpicture}[scale=0.3]\outercycle{2}{1}
\draw[edge_color#1] (x0)--(x1);  
\draw (x0) node[unlabeled_vertex]{};\draw (x1) node[unlabeled_vertex]{};
\labelvertex0
\end{tikzpicture}}
\,
}

\newcommand{\Fuuu}[3]{
\vc{\begin{tikzpicture}[scale=0.4]\outercycle{3}{1}
\draw[edge_color#1] (x0)--(x1);\draw[edge_color#2] (x0)--(x2);  \draw[edge_color#3] (x1)--(x2);    
\draw (x0) node[unlabeled_vertex]{};\draw (x1) node[unlabeled_vertex]{};\draw (x2) node[unlabeled_vertex]{};
\labelvertex0
\end{tikzpicture}}}

\newcommand{\FfourEdges}[6]{
\draw[edge_color#1] (x0)--(x1);\draw[edge_color#2] (x0)--(x2);\draw[edge_color#3] (x0)--(x3);  \draw[edge_color#4] (x1)--(x2);\draw[edge_color#5] (x1)--(x3);  \draw[edge_color#6] (x2)--(x3);
}
\newcommand{\Ffour}[5]{
\vc{\begin{tikzpicture}[scale=0.4]\outercycle{4}{2}
\FfourEdges#5
\draw (x0) node[vertex_#1]{};\draw (x1) node[vertex_#2]{};\draw (x2) node[vertex_#3]{};\draw (x3) node[vertex_#4]{};
\ifthenelse{\equal{#1}{l}}{\labelvertex{0}}{}%
\ifthenelse{\equal{#2}{l}}{\labelvertex{1}}{}%
\ifthenelse{\equal{#3}{l}}{\labelvertex{2}}{}%
\ifthenelse{\equal{#4}{l}}{\labelvertex{3}}{}%
\end{tikzpicture}}
}

\newcommand{\Flluu}[6]{\Ffour{l}{l}{u}{u}{#1#2#3#4#5#6}}
\newcommand{\Flllu}[6]{\Ffour{l}{l}{l}{u}{#1#2#3#4#5#6}}

\newcommand{\Ffive}[6]{
\vc{\begin{tikzpicture}[scale=0.4]\outercycle{5}{4}
#6
\draw (x0) node[vertex_#1]{};
\draw (x1) node[vertex_#2]{};
\draw (x2) node[vertex_#3]{};
\draw (x3) node[vertex_#4]{};
\draw (x4) node[vertex_#5]{};
\ifthenelse{\equal{#1}{l}}{\labelvertex{0}}{}%
\ifthenelse{\equal{#2}{l}}{\labelvertex{1}}{}%
\ifthenelse{\equal{#3}{l}}{\labelvertex{2}}{}%
\ifthenelse{\equal{#4}{l}}{\labelvertex{3}}{}%
\ifthenelse{\equal{#5}{l}}{\labelvertex{4}}{}%
\end{tikzpicture}}
}

\usepackage{xstring}

\newcommand{\WAwl}[2]{%
  \begingroup
  \edef\temp{\detokenize{#1}}%
  \StrSubstitute{\temp}{+}{\%2B}[\temp]%
  \StrSubstitute{\temp}{\\}{\%5C}[\temp]%
  \href{https://www.wolframalpha.com/input?i=\temp}{#2}%
  \endgroup
}

\newtheorem{theorem}{Theorem}

\newtheorem{lemma}[theorem]{Lemma}

\newtheorem{problem}[theorem]{Problem}

\newtheorem{conj}[theorem]{Conjecture}

\newcommand{\pind}{p}

\title{Local maximum of inducibility profiles}
\author{
J\'ozsef Balogh\thanks{Department of Mathematics, University of Illinois Urbana-Champaign, Urbana, IL, USA, and Extremal Combinatorics and Probability Group (ECOPRO), Institute for Basic Science (IBS), Daejeon, South Korea. Email: \texttt{jobal@illinois.edu}. Supported by NSF grants RTG DMS-1937241, FRG DMS-2152488, 
UIUC  Campus Research Board Award RB26026,
the  Simons Collaboration grant SFI-MPS-TSM-00013107, and the Institute for Basic Science (IBS-R029-C4).}
\and
Bernard Lidick\'{y}\thanks{Department of Mathematics, Iowa State University, Ames, IA. E-mail: {\tt lidicky@iastate.edu}. Research of this author is supported in part by NSF FRG DMS-2152490, the Simons Foundation TSM-00013439 and Scott Hanna Professorship.}
\and
Haoran Luo\thanks{Department of Mathematics, Statistics, and Computer Science, University of Illinois Chicago, IL USA. E-mail: {\tt haoranl8@uic.edu}. Research was partially performed while Luo was at the University of Illinois Urbana-Champaign.} 
}
\date{\today}

\begin{document}

\maketitle

\begin{abstract}
 For a graph $G$ and $e\in [0,1]$,
denote by $I_G(e)$ the supremum of induced density of $G$ over $n$-vertex graphs with edge density $e$ as $n$ goes to infinity.
Liu, Mubayi and Reiher asked if there exists a graph $G$, where  $I_G(e)$ has a non-trivial local maximum.

In this paper, we answer this problem affirmatively. We first show that $I_{K_{2,2,1}}(e)$ has at least two local maxima in $(0,1)$. Part of this proof is using flag algebras.
Additionally, we determine 
$I_{K_{2,2,1}}(e)$, when $e=(k-1)/k$ for every integer $k\ge 3.$

We also prove that $I_{K_t^-}(e)$ has a non-global local maximum  for every $t\in\{5,8,11,\ldots,$ $74\}$. The proof combines a symmetrization theorem of Schelp and Thomason with Reiher's clique density theorem.
\end{abstract}

\section{Introduction}

For graphs $G$ and $H$ denote by $I(H,G)$ the number of induced subgraphs of $G$ isomorphic to $H$. The \emph{density} of $H$ in $G$ is 
\[
p(H,G) \coloneqq   \frac{I(H,G)}{\binom{|V(G)|}{|V(H)|}}.
\]
For a graph $H$ and $e \in [0,1]$ let
\begin{align*}
I_H(e) = \limsup_{n\to\infty} \max\{p(H,G) : G \in n\text{-vertex  graphs, } p(K_2,G)=e+o(1)\},\\
i_H(e) = \liminf_{n\to\infty} \min\{p(H,G) : G \in n\text{-vertex  graphs, } p(K_2,G)=e+o(1)\}.
\end{align*}

The maximum of $I_H(e)$ is  the \emph{inducibility} of $H$.
The study of inducibility was initiated by Pippenger and Golumbic~\cite{Pippenger1975} in the 1970s. 
The invention of flag algebras by Razborov~\cite{Razborov} had a major influence on extremal combinatorics, in particular, 
flag algebras have been used to study inducibility as well as in various other settings~\cite{chen2025,bodnar2025,EvenZohar2014,bodnar2025exactinducibility,MR4624284,brosch2024gettingrootproblemsums,2023_KiemPokuttaSpiegel_4colorramsey,Choi2020,Bennett2019,Boyk2022,FalgasRavry2023}.
In this paper we assume the reader is somewhat  familiar with flag algebras.

Liu, Mubayi and  Reiher~\cite{Liu2023feasible} initiated the systematic investigation of $I_H(e)$ and $i_H(e)$.
The \emph{feasible region} $\Omega_{ind}(H)$ of a graph $H$ is the collection of all points $(x,y) \in [0,1]^2$ such that there exists a sequence of graphs of increasing orders $(G_n)$ such that 
$\lim_{n \to \infty} p(K_2,G_n) = x$ and 
$\lim_{n \to \infty} p(H,G_n) = y$.
Such a sequence is called \emph{$H$-good}.
They investigated properties of $\Omega_{ind}(H)$ and proved a more general version of the following theorem.

\begin{theorem}[Liu, Mubayi, Reiher \cite{Liu2023feasible}]\label{thm:LMR} For every  graph $H$ we have
\[
\Omega_{ind}(H) = \{ (x,y) \in [0,1]^2 : i_H(x) \leq y \leq I_H(x) \}.
\]
Moreover, the boundary functions $i_H(x)$ and $I_H(x)$ are continuous and almost everywhere differentiable.    
\end{theorem}

The feasible region is known  only for very few graphs. 
Notable examples are complete graphs~\cite{MR2433944,Reiher2016,MR4089395,Nikiforov2010}.
In their quest to understand graph profiles, Liu, Mubayi and Reiher posed the following problem.

\begin{problem}[Liu, Mubayi, Reiher~\cite{Liu2023feasible}] \label{pro::localMax}
Is there a graph $H$ such that $I_H(e)$ has a non-trivial local maximum?
\end{problem}

For every graph $H$ that is not a clique or an independent set, it is natural to expect that $I_H(e)$ has a roughly unimodal shape: When $e$ is small, there are too few edges to create many induced
copies of $H$, while when $e$ is close to $1$, there are too few non-edges. Problem~\ref{pro::localMax} asks whether this simple single-peak picture
can fail. An affirmative answer would show that, at some edge density strictly smaller than the global optimal density, adding more edges can surprisingly force the largest possible induced density of $H$ to decrease.

In this paper, we answer Problem~\ref{pro::localMax} affirmatively. 

Our first example is $K_{2,2,1}$, the complete $3$-partite graph on $5$ vertices, with class sizes $2,2,1.$

\begin{theorem}\label{thm:main}
The function $I_{K_{2,2,1}}(e)$ has at least two local maxima on $(0,1)$.
\end{theorem}
The inducibility  of $K_{2,2,1}$, with additional stability results, was determined by Pikhurko, Slia\v{c}an and Tyros~\cite[Theorem 1.7.]{Pikhurko2019}.
They proved $I_{K_{2,2,1}}(2/3) = 10/27$, which is achieved by the balanced complete $3$-partite graph.
The inducibility of $K_{2,2,1}$ can also be derived from the earlier work of Brown and Sidorenko~\cite{Brown1994} and Schelp and Thomason~\cite{Schelp1998}
or the later work of Liu, Pikhurko and Staden~\cite{Liu2023}.

The edgeless graphs, balanced complete  $3$-partite and $4$-partite graphs, and the complete graphs imply
$I_{K_{2,2,1}}(0)=0$, $I_{K_{2,2,1}}(2/3)=10/27$, $I_{K_{2,2,1}}(3/4) \geq 45/128$ and $I_{K_{2,2,1}}(1)=0$, respectively. The following lemma adds one more bound to this list.

\begin{lemma}\label{lem:0.74}
  $I_{K_{2,2,1}}(0.74) \leq 44.95/128$.
\end{lemma}

These values for  $I_{K_{2,2,1}}$ together with continuity guaranteed by Theorem~\ref{thm:LMR} imply that 
$I_{K_{2,2,1}}(e)$ has a local maximum in each of $(0,0.74)$ and $(0.74,1)$,  which yields Theorem~\ref{thm:main}.

Next we describe what we suspect to be $I_{K_{2,2,1}}(e)$ for $e \in [0,1]$, see Figure~\ref{fig:k221} for an illustration. 
While $I_{K_{2,2,1}}(e)$ is continuous, it seems to not be differentiable at $e=(k-1)/k$. Our experiments with other multipartite graphs indicate that these peaks may become sharper and may create local maxima.

\begin{conj}\label{conj:K221}
\[
I_{K_{2,2,1}}(e) =
\begin{cases}
\sqrt{25/24}\cdot e^{2.5}
&\quad \text{if } e \in [0,2/3], \\[6pt]
[15 a^{5}(k-2) 
 + 30 (ak-1)^{2} a^{3} 
 - 15 (ak-1)a^{4}](k-1)k
& \quad \text{if } e \in \left[\frac{k-1}{k}, \frac{k}{k+1}\right],
\end{cases}
\]
 where $k \geq 3$ and $a=\frac{k + \sqrt{{\left(1-e \right)} k^{2} - e k}}{k^{2} + k}$.
\end{conj}

Denote the complete $k$-partite $n$-vertex graph, known as the \emph{Tur\'an graph}, by $T_k(n)$.
A graph obtained by taking a disjoint union of graphs $G_1$ and $G_2$ is denoted by $G_1 \cup G_2$. The \emph{joint} of $G_1$ and $G_2$, denoted by $G_1 + G_2$, is obtained from $G_1 \cup G_2$ by adding all edges $uv$, where $u \in V(G_1)$ and $v \in V(G_2)$.

The following constructions on $n$-vertex graphs match the bounds in Conjecture~\ref{conj:K221} as $n$ goes to infinity.
\begin{itemize}
    \item $e \in [0,2/3]$:   $T_3(an) \cup \overline{K_{n-an}}$,\  where   $a=\sqrt{3e/2}$.
    \item $e \in [(k-1)/k, k/(k+1)]$:  $T_k(kan) + \overline{K_{n-kan}}$,\
where
$k \geq 3$
and 
$a=\frac{k + \sqrt{{\left(1-e \right)} k^{2} - e k}}{k^{2} + k}$. 
\end{itemize}
See Figure~\ref{fig:conjK221} for illustrations of these constructions in several regimes.
The following result provides supporting evidence for Conjecture~\ref{conj:K221}.

\begin{theorem}\label{lem:k}
For every real number $k\ge3$,
\[
I_{K_{2,2,1}}\left(\frac{k-1}{k}\right)
\le
\frac{15(k-1)(k-2)}{k^4}.
\]
Equality holds whenever $k$ is an integer.
\end{theorem}

{\bf Remark.} 
Theorem~\ref{lem:k} is not best possible when $k$ is not an integer. It is likely that analyzing the inequalities could slightly improve the upper bound, but it seems that some new ideas are needed to prove that the extremal graphs are complete multipartite. Even proving  Lemma~\ref{lem:0.74} would need some luck.
The result likely could be extended for larger complete $3$-partite graphs, at least for larger edge-densities. Here, one restriction would come from~\cite{balogh2026}, from where we used a tool, for the general case only weaker results are known.

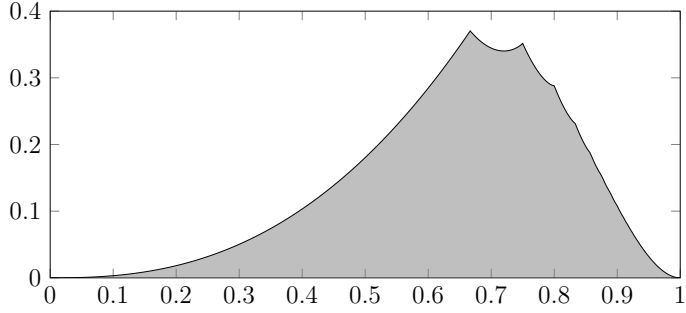
\begin{figure}
\begin{center}    
\begin{tikzpicture}[scale=0.8]
\begin{axis}[ymin=0,ymax=0.4,enlargelimits=false,
width=12cm,height=6cm]
    \addplot [black,fill=gray,fill opacity=0.5,
    ] coordinates {
(0.000000,0.000000) (0.006667,0.000004) (0.013333,0.000021) (0.020000,0.000058) (0.026667,0.000119) (0.033333,0.000207) (0.040000,0.000327) (0.046667,0.000480) (0.053333,0.000670) (0.060000,0.000900) (0.066667,0.001171) (0.073333,0.001486) (0.080000,0.001848) (0.086667,0.002257) (0.093333,0.002716) (0.100000,0.003227) (0.106667,0.003793) (0.113333,0.004413) (0.120000,0.005091) (0.126667,0.005828) (0.133333,0.006625) (0.140000,0.007485) (0.146667,0.008408) (0.153333,0.009396) (0.160000,0.010451) (0.166667,0.011574) (0.173333,0.012766) (0.180000,0.014030) (0.186667,0.015365) (0.193333,0.016774) (0.200000,0.018257) (0.206667,0.019817) (0.213333,0.021454) (0.220000,0.023170) (0.226667,0.024965) (0.233333,0.026841) (0.240000,0.028800) (0.246667,0.030842) (0.253333,0.032968) (0.260000,0.035180) (0.266667,0.037479) (0.273333,0.039865) (0.280000,0.042341) (0.286667,0.044906) (0.293333,0.047563) (0.300000,0.050312) (0.306667,0.053153) (0.313333,0.056089) (0.320000,0.059121) (0.326667,0.062248) (0.333333,0.065473) (0.340000,0.068796) (0.346667,0.072218) (0.353333,0.075740) (0.360000,0.079363) (0.366667,0.083089) (0.373333,0.086917) (0.380000,0.090850) (0.386667,0.094887) (0.393333,0.099030) (0.400000,0.103280) (0.406667,0.107637) (0.413333,0.112103) (0.420000,0.116678) (0.426667,0.121363) (0.433333,0.126159) (0.440000,0.131068) (0.446667,0.136089) (0.453333,0.141224) (0.460000,0.146473) (0.466667,0.151838) (0.473333,0.157319) (0.480000,0.162917) (0.486667,0.168633) (0.493333,0.174468) (0.500000,0.180422) (0.506667,0.186496) (0.513333,0.192692) (0.520000,0.199009) (0.526667,0.205449) (0.533333,0.212012) (0.540000,0.218700) (0.546667,0.225513) (0.553333,0.232451) (0.560000,0.239516) (0.566667,0.246708) (0.573333,0.254028) (0.580000,0.261478) (0.586667,0.269056) (0.593333,0.276765) (0.600000,0.284605) (0.606667,0.292577) (0.613333,0.300681) (0.620000,0.308918) (0.626667,0.317290) (0.633333,0.325796) (0.640000,0.334437) (0.646667,0.343214) (0.653333,0.352129) (0.660000,0.361180) (0.666667,0.370370) (0.670000,0.366778) (0.673333,0.363408) (0.676667,0.360262) (0.680000,0.357340) (0.683333,0.354643) (0.686667,0.352172) (0.690000,0.349928) (0.693333,0.347912) (0.696667,0.346125) (0.700000,0.344570) (0.703333,0.343247) (0.706667,0.342158) (0.710000,0.341307) (0.713333,0.340694) (0.716667,0.340324) (0.720000,0.340200) (0.723333,0.340326) (0.726667,0.340707) (0.730000,0.341350) (0.733333,0.342262) (0.736667,0.343455) (0.740000,0.344944) (0.743333,0.346751) (0.746667,0.348919) (0.750000,0.351562) (0.753333,0.344696) (0.756667,0.338162) (0.760000,0.331965) (0.763333,0.326112) (0.766667,0.320607) (0.770000,0.315460) (0.773333,0.310678) (0.776667,0.306272) (0.780000,0.302255) (0.783333,0.298644) (0.786667,0.295459) (0.790000,0.292730) (0.793333,0.290502) (0.796667,0.288852) (0.800000,0.288000) (0.803333,0.280202) (0.806667,0.272815) (0.810000,0.265851) (0.813333,0.259328) (0.816667,0.253266) (0.820000,0.247689) (0.823333,0.242635) (0.826667,0.238157) (0.830000,0.234350) (0.833333,0.231481) (0.836667,0.223609) (0.840000,0.216208) (0.843333,0.209307) (0.846667,0.202943) (0.850000,0.197169) (0.853333,0.192078) (0.856667,0.187894) (0.860000,0.180858) (0.863333,0.173684) (0.866667,0.167077) (0.870000,0.161113) (0.873333,0.155943) (0.876667,0.150115) (0.880000,0.143142) (0.883333,0.136782) (0.886667,0.131166) (0.890000,0.125682) (0.893333,0.119029) (0.896667,0.113042) (0.900000,0.108000) (0.903333,0.101546) (0.906667,0.095776) (0.910000,0.090487) (0.913333,0.084591) (0.916667,0.079572) (0.920000,0.073846) (0.923333,0.068886) (0.926667,0.063541) (0.930000,0.058636) (0.933333,0.053926) (0.936667,0.049105) (0.940000,0.044558) (0.943333,0.040198) (0.946667,0.036020) (0.950000,0.032063) (0.953333,0.028196) (0.956667,0.024604) (0.960000,0.021197) (0.963333,0.017991) (0.966667,0.015037) (0.970000,0.012303) (0.973333,0.009824) (0.976667,0.007602) (0.980000,0.005645) (0.983333,0.003961) (0.986667,0.002561) (0.990000,0.001455) (0.993333,0.000653) (0.996667,0.000165) (1,0)
    }
    ;
\end{axis}
\end{tikzpicture}
\end{center}
\caption{Lower bound on $I_{K_{2,2,1}}(e)$ from constructions.}\label{fig:k221}
\end{figure}

\tikzset{myscale/.style={scale=0.9}}
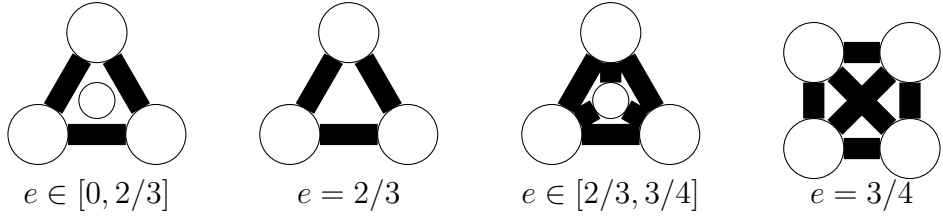
\begin{figure}
\begin{center}
    \begin{tikzpicture}[myscale]
        \draw(0,0) 
        node[big_turan_vertex_center,scale=0.8](x0){}
        (90:1) node[big_turan_vertex](x1){}
        (210:1) node[big_turan_vertex](x2){}
        (330:1) node[big_turan_vertex](x3){}
        ;
        \draw[line width=8pt, color=black] (x1)--(x2)--(x3)--(x1);
        \draw (0,-1) node[below]{$e \in [0,2/3]$};
    \end{tikzpicture}
\hskip 2em
    \begin{tikzpicture}[myscale]
        \draw
        (90:1) node[big_turan_vertex](x1){}
        (210:1) node[big_turan_vertex](x2){}
        (330:1) node[big_turan_vertex](x3){}
        ;
        \draw[line width=8pt, color=black] (x1)--(x2)--(x3)--(x1);
        \draw (0,-1) node[below]{$e = 2/3$};
    \end{tikzpicture}
\hskip 2em   
    \begin{tikzpicture}[myscale]
        \draw(0,0) 
        node[big_turan_vertex_center,scale=0.8](x0){}
        (90:1) node[big_turan_vertex](x1){}
        (210:1) node[big_turan_vertex](x2){}
        (330:1) node[big_turan_vertex](x3){}
        ;
        \draw[line width=8pt, color=black] (x1)--(x2)--(x3)--(x1) (x0)--(x1) (x0)--(x2) (x0)--(x3);
        \draw (0,-1) node[below]{$e \in [2/3,3/4]$};
    \end{tikzpicture} 
\hskip 2em
     \begin{tikzpicture}[myscale]
        \draw
        (45:1) node[big_turan_vertex](x1){}
        (135:1) node[big_turan_vertex](x2){}
        (225:1) node[big_turan_vertex](x3){}
        (-45:1) node[big_turan_vertex](x4){}
        ;
        \draw[line width=8pt, color=black] (x1)--(x2)--(x3)--(x1) (x4)--(x1) (x4)--(x2) (x4)--(x3);
        \draw (0,-1) node[below]{$e = 3/4$};
    \end{tikzpicture}
\end{center}
    \caption{Constructions for Conjecture~\ref{conj:K221}.}
    \label{fig:conjK221}
\end{figure}

We next give a second example, based on cliques with one edge removed.
For $t\ge3$, let $K_t^-$ denote the graph obtained from $K_t$ by
deleting one edge.
We show that $K_t^-$ has a local maximum for several values of $t$.

\begin{theorem}\label{thm:ktminus-local}
$I_{K_t^-}(x)$ has a strict non-global local maximum for all $t \in \{5,8,11,\ldots,74\}$.
\end{theorem}

The proof combines a
Schelp--Thomason reduction to balanced complete multipartite graphons
\cite{Schelp1998} with the clique density theorem by Reiher~\cite{Reiher2016}.

While we confirm the existence of graphs which have at least two local maxima, we suspect that the number of local maxima of an inducibility profile is unbounded in general.
Let $\mathcal{M}$ be the family of complete multipartite graphs. 

\begin{conj}
    For every $k > 0$, there is a graph $H \in \mathcal{M}$, where $I_H(x)$ has at least $k$ local maxima. 
\end{conj}

The inducibility of balanced graphs in $\mathcal{M}$ was determined by Bollob\'as, Egawa, Harris and Jin~\cite{Bollobs1995} and a more general result for inducibility of a linear combination of graphs from $\mathcal{M}$ was obtained by Schelp and Thomason~\cite{Schelp1998}.

In Section~\ref{extsec}, we prove Theorem~\ref{lem:k}, and in Section~\ref{mainsec}, we prove Theorem~\ref{thm:main}. In Section~\ref{sec:larger-t}, we prove Theorem~\ref{thm:ktminus-local}.
In the Appendices, we provide further details of the flag algebra and computer-assisted calculations.

\section{Proof of Theorem~\ref{lem:k}}\label{extsec}

\begin{proof}[Proof of Theorem~\ref{lem:k}]    
We want to show for every real number $k\ge3$,
\[
I_{K_{2,2,1}}\left(\frac{k-1}{k}\right)
\le
\frac{15(k-1)(k-2)}{k^4}.
\]
The lower bound is achieved by  $T_k(n)$.
For the upper bound note that the result is known for $k=3$. In the rest of the proof we assume that $k\ge 3$.

To prove the upper bound, it is enough to consider
an arbitrary sequence $(G_n)$ with $p(K_2,G_n)\to \frac{k-1}{k}$ and prove that
\[
p(K_{2,2,1},G_n)\le \frac{15(k-1)(k-2)}{k^4}+o(1).
\]
By changing $o(n^2)$ edges, we may assume
\[
e(G_n)=\frac{k-1}{k} \cdot \binom n2+O(1).
\]
Indeed, changing $o(n^2)$ edges changes the number of induced copies of
$K_{2,2,1}$ by at most $o(n^5)$. Thus, for readability, we write
$G=G_n$ and suppress lower-order terms throughout the calculation.

As a warm-up, we first describe the proof for regular graphs. Notice that the number of edges of $G$ is 
$$ e(G)= n^2\cdot\frac{k-1}{2k}.$$
Denote $d_u$ the degree of the vertex $u$, and $d_{uv}$  the number of common neighbors of $u$ and $v$. 
Denote by $I(u,v)$ the number of induced copies of $K_{2,2,1}$, where $uv$ is an edge between the two larger classes. 
Observe that from an edge $uv$, every induced $K_{2,2,1}$ containing $uv$ between the two larger classes,  can be built  by choosing a common neighbor $x$, a $u'\in N(u)-N(v)$, and a $v'\in N(v)-N(u)$, but not all those choices work, as some of $v'u', xv', xu'$ might not be an edge. 
Writing $d=(k-1)n/k$, 
 and assuming that $n$ is a multiple of $k$,
 and using
$$
I(u,v) \leq
\underbrace{|N(u)-N(v)|}_{u'}\cdot \underbrace{| N(v)-N(u)|}_{v'}\cdot  \underbrace{|N(v)\cap N(u)|}_{x} = (d-d_{uv})^2d_{uv}\le (n-d)^2\cdot (2d-n), $$
where we use the fact that when $d\ge 2n/3$, the expression is optimal when $d_{uv}$ is minimized.
Now we have
\begin{equation}\label{goale}
4\cdot \binom{n}{5}\cdot 
p(K_{2,2,1},G)
=
\sum_{uv \in E(G)} I(u,v) 
\le \underbrace{n^2\cdot\frac{k-1}{2k}}_{=e(G)} \frac{n^2}{k^2} \frac{(k-2)n}{k},
\end{equation}
which implies the required inequality from the statement of the theorem.

Now we extend this argument to the general case.
We prove the upper bound in the following way
\begin{align}\nonumber
4\cdot \binom{n}{5}\cdot 
p(K_{2,2,1},G)
=
\sum_{uv\in E(G)}I(u,v)
&\le
\sum_{uv\in E(G)}(d_u-d_{uv})(d_v-d_{uv})\cdot d_{uv}
\\ 
\label{semi}
&\le 
\underbrace{n^2\cdot\frac{k-1}{2k}}_{=e(G)} \frac{(k-2)n^3}{k^3}.
\end{align}
The rest of the proof is proving the last inequality in \eqref{semi}.
We start by adding zero to the left-hand side:

\begin{align*}
& \sum_{uv\in E(G)}I(u,v) \leq 
\sum_{uv\in E(G)}(d_u-d_{uv})(d_v-d_{uv})\cdot d_{uv} \\ 
=&\sum_{uv\in E(G)}(d_u-d_{uv})(d_v-d_{uv})\cdot d_{uv} - \frac{k-3}{k}n(n-d_u)(n-d_v) 
 + \frac{k-3}{k}n(n-d_u)(n-d_v).
\end{align*}
Theorem 2.2 in~\cite{balogh2026} implies that the left-hand side of \eqref{eq:Balogh} is maximized by regular graphs,  giving
\begin{equation}\label{eq:Balogh}
\sum_{uv\in E(G)}(n-d_u)(n-d_v)\leq 
\frac{k-1}{2k^3}\cdot n^4.
\end{equation}
This gives 
\begin{equation}
\sum_{uv\in E(G)}\frac{k-3}{k}n(n-d_u)(n-d_v)\le  \frac{(k-1)(k-3)}{2k^4} \cdot n^5
=
\underbrace{n^2\cdot\frac{k-1}{2k}}_{=e(G)} \frac{k-3}{k^3}\cdot n^3.
\end{equation}
Hence, it is sufficient to prove:
\begin{equation}
\sum_{uv\in E(G)}(d_u-d_{uv})(d_v-d_{uv})\cdot d_{uv} - \frac{k-3}{k}n(n-d_u)(n-d_v) \le \underbrace{n^2\cdot\frac{(k-1)}{2k}}_{=e(G)}\cdot 
\frac{n^3}{k^3}.
\end{equation}
Our goal is to upper bound each term in the sum above by $\frac{n^3}{k^3}$.
Fix an arbitrary $uv\in E(G)$. We want to show:
\begin{equation}
\label{eq:want}
(d_u-d_{uv})(d_v-d_{uv})\cdot d_{uv} - \frac{k-3}{k}n(n-d_u)(n-d_v) \le 
\frac{n^3}{k^3}.
\end{equation}

In the following optimization, we vary the numerical variables $d_u,d_v,d_{uv}$ subject to the necessary inequalities they satisfy and prove our claim, this is not meant to correspond to modifying the graph $G$.

{\bf Case (i)}: Assume 
$2n/3 \leq d_u \leq d_v$. 
We first prove the  following inequality:
\begin{align}\label{eq:2/3}
(d_u-d_{uv})(d_v-d_{uv})\cdot d_{uv}\le (n-d_u)(n-d_v)(d_u+d_v-n).
\end{align}
Observe $d_u - d_{uv} \leq n-d_v$,
 and since $2n/3 \leq d_u$, we have  $d_{uv} \geq d_u/2$. 
Hence for a fixed $d_u$, $(d_u-d_{uv})d_{uv}$ is maximized when $d_{uv}$ is minimized. 
Moreover, $d_v - d_{uv}$ is also maximized when $d_{uv}$ is minimized.
As $d_u+d_v-n\le d_{uv}$, the  left-hand side of\eqref{eq:2/3} is maximized when
$d_u+d_v-n= d_{uv}$, which implies that the relation \eqref{eq:2/3} holds.

\noindent Now we prove \eqref{eq:want}, using   \eqref{eq:2/3}
and applying the arithmetic mean-geometric mean inequality:\footnote{$x_1x_2x_3 \leq ((x_1+x_2+x_3)/3 )^3$.} 
\begin{align*}
&~ (d_u-d_{uv})(d_v-d_{uv})\cdot d_{uv} - \frac{k-3}{k}n(n-d_u)(n-d_v) \\
&\le
\left(d_u+d_v-n- \frac{k-3}{k}n\right)(n-d_u)(n-d_v)  \leq \frac{n^3}{k^3}.
\end{align*}

{\bf Case (ii)}:
Assume $d_u \leq d_v$ and $d_u \leq 2n/3$. 
We want to show~\eqref{eq:want}, that we rearrange as:

\begin{equation}\label{eq:case2}
(d_u-d_{uv})(d_v-d_{uv})\cdot d_{uv} \le  \frac{k-3}{k}n(n-d_u)(n-d_v) + 
\frac{n^3}{k^3}.
\end{equation}

Observe that $d_u-d_{uv} \le n-d_v$,
$d_v-d_{uv} \le n-d_u$,
and $d_{uv} \leq d_u \leq 2n/3$.
We are done if $d_{uv} \leq ((k-3)/k + 1/k^3)n.$
That is already the case for $k\ge 9$ when $(k-3)/k+1/k^3 > 2/3$.
For fixed $d_u$ and $d_v$, we aim to  find the value of $d_{uv}$ when the left-hand side is maximized.
If $d_{uv} > d_u/2$, then decreasing $d_{uv}$ to $d_u/2$ increases  $(d_u-d_{uv})d_{uv}$ as well as $(d_v - d_{uv})$,
meaning that the maximum  happens when $d_{uv}\le d_u/2< n/3$.

Observe that for given $d_u, d_{uv}$ the worst case is when $d_v$ is maximized, but then $d_u+d_v=n+ d_{uv}$. 
Writing it into~\eqref{eq:case2} we obtain after rearranging 

\begin{equation}
  \left( \frac{3-2k}{k}n + d_u+d_v\right) (n-d_u)(n-d_v) \le 
\frac{n^3}{k^3},
\end{equation}
where the last inequality holds by the 
arithmetic mean-geometric mean inequality.
This completes the proof of \eqref{semi}.
\end{proof}

An alternative proof of Theorem~\ref{lem:k} using flag algebras is in Appendix A.

\section{Proof of Theorem~\ref{thm:main}}\label{mainsec}

First we show $I_{K_{2,2,1}}(0.74) < 44.95/128$ to prove Lemma~\ref{lem:0.74}. 
The goal of Lemma~\ref{lem:0.74} is to provide a calculation that fits in the paper and gives a bound strictly less than  $I_{K_{2,2,1}}(3/4) = 45/128$.
There is a calculation using flag algebras on $7$ vertices for $I_{K_{2,2,1}}(0.74)$ that numerically coincides with the bound from Conjecture~\ref{conj:K221}, however, this computation is large, and  would not fit in this paper.

\begin{proof}[Proof of Lemma~\ref{lem:0.74}]
We denote by $\mathcal{F}_\ell$ the set of (unlabelled)  $\ell$-vertex graphs.
Define the following linear combinations of flags.
\begin{align*}
  A &\coloneqq     -0.474 \Fe40{1 1 1 1 1 1} - 0.581 \Fe40{1 1 1 1 1 2} + 3.861 \Fe40{1 1 2 1 2 2} - 0.795 \Fe40{1 1 1 2 2 2} - 0.362 \Fe40{1 1 2 2 2 2} - 1.982 \Fe40{2 2 2 2 2 2} \\
    &~ - 0.766 \Fe40{1 2 2 2 2 2} - 1.235 \Fe40{1 2 2 2 2 1} \\
  B &\coloneqq  -0.97212 \Fe42{2 1 1 2 2 1} + 0.01048 \Fe42{2 2 2 2 2 2} + 0.13324 \Fe42{2 1 2 2 2 2} + 0.00138 \Fe42{2 2 2 1 2 2} + 0.04384 \Fe42{2 2 2 2 2 1} \\&~ + 0.18756 \Fe42{2 1 2 2 1 2} \\
  C &\coloneqq   -0.0986 \Fe42{2 1 1 2 2 1} - 0.356 \Fe42{2 2 2 2 2 2} + 0.175 \Fe42{2 1 2 2 2 2} + 0.1587 \Fe42{2 2 2 1 2 2} + 0.5142 \Fe42{2 2 2 2 2 1} - 0.737 \Fe42{2 1 2 2 1 2}\\
  D &\coloneqq  -0.0922 \Fe42{2 1 1 2 2 1} - 0.003 \Fe42{2 2 2 2 2 2} - 0.6957 \Fe42{2 1 2 2 2 2} + 0.7122 \Fe42{2 2 2 1 2 2} + 0.01 \Fe42{2 2 2 2 2 1} + 0.0089 \Fe42{2 1 2 2 1 2} 
\end{align*}

Assuming that the density of edges is $0.74$,  we have $ \left(  \Fe20{2} - 0.74  \right) = 0$. 
By adding non-negative terms to the density of $K_{2,2,1}$, we obtain the desired upper bound as
 \begin{align*}
  \Fe50{2 2 2 1 1 2 2 2 2 2}  &\leq   \Fe50{2 2 2 1 1 2 2 2 2 2} 
   + A \times  \left(  \Fe20{2} - 0.74  \right) 
   + 14.509\cdot  \left\llbracket B  ^{2} \right\rrbracket  
   + 6.822\cdot  \left\llbracket  C ^{2} \right\rrbracket 
   + 0.444\cdot  \left\llbracket  D^{2} \right\rrbracket   \\
   &= \sum_{F \in \mathcal{F}_6} c_F F < \frac{44.95}{128} \sum_{F \in \mathcal{F}_6} F =\frac{44.95}{128}.
 \end{align*}
Since $|\mathcal{F}_6|=156$ and the coefficients $c_F$ are numerical, we list them in the Appendix~B. \end{proof}

\begin{proof}[Proof of Theorem~\ref{thm:main}]
It follows from Lemma~\ref{lem:0.74} and Theorem~\ref{lem:k} that \quad 
\[
 I_{K_{2,2,1}}(0) = 0, \quad    I_{K_{2,2,1}}(2/3) = 10/27>44.95/128, \]
 \[
I_{K_{2,2,1}}(0.74) < 44.95/128, \quad 
I_{K_{2,2,1}}(0.75) = 45/128, \quad 
I_{K_{2,2,1}}(1) = 0.
\]
In particular, we know that 
$I_{K_{2,2,1}}(2/3) = 10/27$ is a global maximum, and $I_{K_{2,2,1}}(0.74) < 
I_{K_{2,2,1}}(0.75)$, i.e., there is another local maximum which is at least $0.74$.
\end{proof}

\section{Proof of Theorem~\ref{thm:ktminus-local}}\label{sec:larger-t}

For several values of $t$, we show an upper bound on 
the density of $K_t^-$ that is a linear combination \eqref{eq:homogeneous-certificate} of edge density and densities of two other cliques. 
This upper bound has a local maximum.
The results are summarized in Table~\ref{tab:homogeneous-certificates}.

We first state a lemma implying we need to verify \eqref{eq:homogeneous-certificate} only for complete balanced multipartite setting.

\begin{lemma}\label{lem:balanced-reduction}
Let $\mathcal S$ and $\mathcal R$ be finite sets of positive integers, with $s\ge 3$ for every $s\in\mathcal S$. For each $s\in\mathcal S$, let $\alpha_s>0$, and for each $r\in\mathcal R$, let $\beta_r\in\mathbb R$. Then, for every $n$, the quantity
\[
\sum_{s\in\mathcal S}\alpha_s I(K_s^-,G)
+
\sum_{r\in\mathcal R}\beta_r I(K_r,G)
\]
is maximized, among all $n$-vertex graphs $G$, by a balanced complete multipartite graph.
\end{lemma}

\begin{proof}
We use Theorem~1 of Schelp and Thomason~\cite{Schelp1998}. In the form needed here, it says that if
\[
f(G)=\sum_F c_F I(F,G),
\]
where each $F$ is complete multipartite, and where $c_F>0$ unless $F$ is a complete graph, in which case $c_F$ may be arbitrary, then $f(G)$ attains its maximum over all $n$-vertex graphs on a complete multipartite graph.
The graphs $K_s^-=K_{2,1,\dots,1}$ are complete multipartite and have positive coefficients $\alpha_s$, while the graphs $K_r$ are complete graphs and have arbitrary coefficients $\beta_r$. Therefore the  function $f(G)$ is maximized by some complete multipartite graph.
It remains to show that, among complete multipartite graphs, a maximizer may be chosen balanced. Let the part sizes be
\[
n_1,\dots,n_q,\qquad n_i\ge 1,\qquad \sum_{i=1}^q n_i=n.
\]
Choose a maximizing complete multipartite graph with the minimum possible number of parts.
We claim that any two positive parts have sizes differing by at most one. Fix all parts except two of them, say $x$ and $y$, and write
\[
h=x+y,\qquad p=xy.
\]
We show that, after all other part sizes and $h$ are fixed, the whole functional is affine in $p$.

For the  clique terms this is immediate. If $E_j$ denotes the $j$th elementary symmetric polynomial in all part sizes except $x$ and $y$, then
\[
e_r=E_r+hE_{r-1}+pE_{r-2}.
\]
Since the number of copies of $K_r$ in a complete multipartite graph is $e_r(n_1,\dots,n_q)$, each clique term is affine in $p$.

Now consider an induced $K_s^-$-term. In a complete multipartite graph, an induced copy of $K_s^-=K_{2,1,\dots,1}$ is obtained by choosing two vertices from one part and one vertex from each of $s-2$ other parts. Hence its count is
\[
\sum_i \binom{n_i}{2} e_{s-2}(n_1,\dots,\widehat{n_i},\dots,n_q).
\]
The contribution from parts other than $x,y$ is affine in $p$, again by
\[
e_j=E_j+hE_{j-1}+pE_{j-2}.
\]
The contribution from the two special parts is
\[
\binom{x}{2}\left(E_{s-2}+yE_{s-3}\right)
+
\binom{y}{2}\left(E_{s-2}+xE_{s-3}\right),
\]
where now the $E_j$'s are taken over all parts other than $x,y$. Since
\[
\WAwl{\binom{x}{2}+\binom{y}{2} - ( \frac{(x+y)*(x+y-1)}{2}-x*y )}{
\binom{x}{2}+\binom{y}{2}
=
\frac{h(h-1)}2-p
}
\quad \quad 
\text{and} \quad \quad 
\WAwl{y*\binom{x}{2}+x*\binom{y}{2} -  (x+y-2)*x*y/2}{
y\binom{x}{2}+x\binom{y}{2}
=
\frac{p(h-2)}2,}
\]
this contribution is also affine in $p$. Hence the whole functional is affine in $p$.

Therefore, for fixed $h=x+y$, the functional has the form
\[
A+Bp
\]
as a function of $p=xy$. If $B<0$, then replacing $(x,y)$ by $(h,0)$ strictly increases the value, contradicting maximality. If $B=0$, then replacing $(x,y)$ by $(h,0)$ preserves the value and decreases the number of positive parts, contradicting the minimum choice of the support. Hence $B>0$.
Thus $p=xy$ must be as large as possible among integer pairs
\[
x,y\ge 0,\qquad x+y=h.
\]
This happens exactly when
\[
|x-y|\le 1.
\]
Since the pair $x,y$ was arbitrary, all positive part sizes differ by at most one. Therefore the complete multipartite maximizer is balanced, namely it is a $T_q(n)$ for some $q$.
\end{proof}

The local maxima are obtained from inequalities
\begin{equation}\label{eq:homogeneous-certificate}
\pind(K_t^-,W)
+\mu_R \pind(K_R,W)
+\mu_{R+1} \pind(K_{R+1},W)
\le
\lambda \pind(K_2,W),
\end{equation}
which we will prove holds for every graphon $W$ for some values of $t$ and suitable values $R$ and $\lambda\coloneqq \mu_R+\mu_{R+1}$.
These values or their numerical approximations are listed in Table~\ref{tab:homogeneous-certificates}.

Lemma~\ref{lem:balanced-reduction} will imply that we need to check \eqref{eq:homogeneous-certificate} only for complete multipartite graphons. 
To do that, we use the following definitions.
For integers $t\ge3$ and $q\ge1$,
let $P_{t,q}$ be the induced density of $K_t^-$ in the balanced
complete $q$-partite graphon $T_q$. 
It can be calculated as
\[
P_{t,q}\coloneqq \binom{t}{2}\frac{(q-1)_{t-2}}{q^{t-1}}
\leq \binom{t}{2}\frac{1}{q},
\]
where $(a)_b=a(a-1)\cdots(a-b+1)$, with the convention that
$(a)_b=0$ when $a$ is a nonnegative integer smaller than $b$.
We also write 
\[
C_r(q)\coloneqq \pind(K_r,T_q)=(q)_r/q^r
\quad\quad \text{ and } \quad\quad 
x_q\coloneqq \pind(K_2,T_q)=1-1/q.
\]
Put $A_r(q) \coloneqq x_q-C_r(q)$. For a row in Table~\ref{tab:homogeneous-certificates} with parameters $(t,k,R)$,
the coefficients $\mu_R$ and $\mu_{R+1}$ are the solution of
\begin{equation}\label{eq:exact-mu-system}
\begin{pmatrix}
A_R(k-1) & A_{R+1}(k-1)\\
A_R(k)   & A_{R+1}(k)
\end{pmatrix}
\begin{pmatrix}
\mu_R\\
\mu_{R+1}
\end{pmatrix}
=
\begin{pmatrix}
P_{t,k-1}\\
P_{t,k}
\end{pmatrix}.
\end{equation}
The accompanying exact computation verifies that this matrix is
nonsingular and that $\mu_R,\mu_{R+1}>0$ in every row. 
All verifications use the
exact rational solutions of \eqref{eq:exact-mu-system}. By construction,
\eqref{eq:homogeneous-certificate} is an equality for both $T_{k-1}$
and $T_k$.
The theorem below is a slightly more precise version of Theorem~\ref{thm:ktminus-local}.

\begin{theorem}\label{thm:larger-t}
For every $t=5+3j$ with $0\le j\le23$, the function $I_{K_t^-}(x)$
has a strict non-global local maximum at $x=1-1/k$, where $k$ is given
in Table~\ref{tab:homogeneous-certificates}.
\end{theorem}

\subsection{Validity of \eqref{eq:homogeneous-certificate}}

In this subsection we show how we verify \eqref{eq:homogeneous-certificate} 
for fixed values of $t$ and $R$ from Table~\ref{tab:homogeneous-certificates}.
Notice $R \geq t-2$ in each row of the table.

For each $n$, Lemma~\ref{lem:balanced-reduction} implies that the graph $G_n$ maximizing
\begin{equation}\label{eq:homogeneous-I}
\frac{1}{\binom{n}{t}}I(K_t^-,G_n)
+\frac{\mu_R}{\binom{n}{R}} I(K_R,G_n)
+\frac{\mu_{R+1}}{\binom{n}{R+1}} I(K_{R+1},G_n)
-\frac{\lambda}{\binom{n}{2}} I(K_2,G_n)
\end{equation}
is complete multipartite.
Passing to graphons, it is therefore enough to verify
\eqref{eq:homogeneous-certificate} on the balanced complete
$q$-partite graphons for all $q$.
Thus we must prove
\[
P_{t,q}+\mu_R C_R(q)+\mu_{R+1}C_{R+1}(q)
- \lambda x_q \leq 0
\]
for every positive integer $q$. 
This is equivalent to showing
\[
S(q)\coloneqq 
\lambda\left(1-\frac1q\right)
-\mu_R\frac{(q)_R}{q^R}
-\mu_{R+1}\frac{(q)_{R+1}}{q^{R+1}}
-P_{t,q}
\]
is nonnegative.

For $L\in \mathbb{N}$ that we define later, we
check $S(q) \geq 0$ all $q < L$ individually. 
For $q \geq L$, we use the following argument.
For $r\ge1$, let
$B_r\coloneqq r(r-1)(3r^2-7r+2)/24$. The second Bonferroni inequality\footnote{For $0\le y_i\le1$, the second Bonferroni inequality gives
$\prod_i(1-y_i)\le
1-\sum_i y_i+\sum_{i<j}y_i y_j$.
We apply this with $y_i=i/q$ for $0\le i\le r-1$.}
gives, for $q\ge r$,
\[
\frac{(q)_r}{q^r}
=
\prod_{i=0}^{r-1}\left(1-\frac{i}{q}\right)
\le
1-\frac{\binom r2}{q}+\frac{B_r}{q^2}.
\]
This together with $P_{t,q}\le\binom t2/q$ gives a lower bound on $S(q)$.
\[
S(q) \geq \lambda - \mu_R -\mu_{R+1}
+ \frac{1}{q}\left(-\lambda + \mu_R\binom{R}{2} + \mu_{R+1}\binom{R+1}{2} - \binom{t}{2}   \right)
-\frac{1}{q^2}\left( \mu_RB_R+\mu_{R+1}B_{R+1} \right).
\]
To shorten the expression, define
\[
\Lambda\coloneqq 
\mu_R\binom R2+\mu_{R+1}\binom{R+1}{2}
-\lambda-\binom t2,
\qquad
D\coloneqq \mu_RB_R+\mu_{R+1}B_{R+1}.
\]
We checked $\Lambda>0$ for every row of Table~\ref{tab:homogeneous-certificates}.
Since
$\lambda=\mu_R+\mu_{R+1}$, we conclude
\[
S(q)\ge\frac{\Lambda}{q}-\frac{D}{q^2}\ge0
\quad\text{whenever}\quad
q\ge
L\coloneqq \max\left\{R+1,\left\lceil\frac{D}{\Lambda}\right\rceil\right\}.
\]

The accompanying SageMath code with exact arithmetic solves \eqref{eq:exact-mu-system}, computes $L$,
checks $\Lambda>0$ and 
verifies $S(q) \geq 0$ for all $q < L$.
This proves \eqref{eq:homogeneous-certificate} for every row of Table~\ref{tab:homogeneous-certificates}.

\subsection{Strict local maximum using derivatives}

Let $g_r(x)$ denote the minimum possible $K_r$-density among graphons
of edge density $x$. By definition, every graphon $W$ with
$\pind(K_2,W)=x$ satisfies
\[
\pind(K_r,W)\ge g_r(x).
\]
Since $\mu_R,\mu_{R+1}>0$, it follows from
\eqref{eq:homogeneous-certificate} that
\[
I_{K_t^-}(x)
\le
U(x)\coloneqq
\lambda x-\mu_Rg_R(x)-\mu_{R+1}g_{R+1}(x).
\]

The clique density theorem determines $g_R$ and $g_{R+1}$ near
$x_k\coloneqq1-1/k$. In particular, the balanced complete
$k$-partite graphon $T_k$ attains the minimum defining both
$g_R(x_k)$ and $g_{R+1}(x_k)$. Since \eqref{eq:homogeneous-certificate} is an equality
on $T_k$, we have
\[
P_{t,k}
=
\pind(K_t^-,T_k)
\le
I_{K_t^-}(x_k)
\le
U(x_k)
=
P_{t,k}.
\]
Therefore
\[
I_{K_t^-}(x_k)=P_{t,k}=U(x_k).
\]

We show that this point is a local maximum by checking one-sided derivatives.
We claim that the one-sided derivatives of the clique
density function at $x_k$ for $k\ge r$ are
\begin{equation} \label{equ::oneSideDeri}
g_r'(x_k-)
=
\binom r2\frac{(k)_r}{(k-1)k^{r-1}},
\qquad
g_r'(x_k+)
=
\binom r2\frac{(k)_{r-1}}{k^{r-1}}.
\end{equation}
To obtain the left derivative, consider the complete $k$-partite
graphon with $k-1$ equal parts of size $a=(1-z)/(k-1)$ and one part of
size $z$. 
Its edge density is
\[
x(z)=1-\frac{(1-z)^2}{k-1}-z^2.
\]
The corresponding $K_r$-density is
\[
h_r(z)=
r!\left[
\binom{k-1}{r}a^r+
\binom{k-1}{r-1}a^{r-1}z
\right].
\]
Then $x(1/k)=x_k$, $h_r(1/k)=g_r(x_k)$,
$
\WAwl{D[1-(1-z)^2/(k-1)-z^2, z] /. z -> 1/k}{x'(1/k)}=
\WAwl{D[r!*(Binomial[k-1,r]*((1-z)/(k-1))^r + Binomial[k-1,r-1]*((1-z)/(k-1))^(r-1)*z), z] /. z = 1/k
}{
h_r'(1/k)}=0$, and
\[
\WAwl{D[1-(1-z)^2/(k-1)-z^2,{z,2}] /. z -> 1/k}{
x''(1/k)=-\frac{2k}{k-1}.
}
\]
Differentiation gives
\[
\WAwl{D[( (k-r)/r  ((1-z)/(k-1))^r + ((1-z)/(k-1))^(r-1)*z), {z,2}]
/. z = 1/k (* r!binom(k-1,r-1) are factored out  *)}{
h_r''(1/k)
=
-\frac{r!\binom{k-1}{r-1}(r-1)}{(k-1)^2}\,k^{3-r}.
}
\]
Therefore
\[
g_r'(x_k-)
=
\frac{h_r''(1/k)}{x''(1/k)}
=
\binom r2\frac{(k)_r}{(k-1)k^{r-1}}.
\]

For the right derivative, let $b=(1-z)/k$ and consider the complete
$(k+1)$-partite graphon with $k$ parts of size $b$ and one part of
size $z$. Its edge density is
\[
x_+(z)=1-\frac{(1-z)^2}{k}-z^2,
\]
and the corresponding $K_r$-density is
\[
h_{r,+}(z)
=
r!\left[
\binom{k}{r}b^r+
\binom{k}{r-1}b^{r-1}z
\right].
\]
At $z=0$, we have $x_+(0)=x_k$, and
\[
\WAwl{D[1-(1-z)^2/k-z^2,z] /. z -> 0}{
x_+'(0)=\frac2k,}
\qquad
\WAwl{D[(k-r+1)/r*((1-z)/k)^r +((1-z)/k)^(r-1)*z, z]
/. z = 0 (* r!binom(k,r-1) are factored out  *)}{
h_{r,+}'(0)=
r!\binom{k}{r-1}(r-1)k^{-r}.}
\]
Thus
\[
g_r'(x_k+)
=
\frac{h_{r,+}'(0)}{x_+'(0)}
=
\binom r2\frac{(k)_{r-1}}{k^{r-1}}.
\]
This proves \eqref{equ::oneSideDeri}.

It follows that
\[
\begin{aligned}
U'(x_k-)
&=
\lambda-\mu_Rg_R'(x_k-)-\mu_{R+1}g_{R+1}'(x_k-),\\
U'(x_k+)
&=
\lambda-\mu_Rg_R'(x_k+)-\mu_{R+1}g_{R+1}'(x_k+).
\end{aligned}
\]
The last two columns of Table~\ref{tab:homogeneous-certificates} give
$U'(x_k-)$ and $U'(x_k+)$, respectively. Since the left one is positive and the right one is negative, $U$ has a strict local maximum at $x_k$. Since
$I_{K_t^-}(x)\le U(x)$ and equality holds at $x_k$, the same point is a
strict local maximum of $I_{K_t^-}$.

It remains to prove that this local maximum is not global.
This is implied by $P_{t,k-1}>P_{t,k}$.
It is possible to show $P_{t,k-1}$ is the global maximum. 

As in the previous subsection, the values $U'(x_k-), U'(x_k+), P_{t,k}$ and $P_{t,k-1}$ were calculated using SageMath. 
The code is available as part of the preprint submission to arXiv~\cite{balogh2026arXiv}.
This finishes the proof of Theorem~\ref{thm:larger-t}.

\begin{table}[htbp]
\centering
\scriptsize
\renewcommand{\arraystretch}{1.25}
\resizebox{\textwidth}{!}{
\begin{tabular}{c c c c c c c c c}
\hline
$t$ & $k$ & $R$
& $\mu_R$ & $\mu_{R+1}$
& \shortstack{local\\value $P_{t,k}$}
& \shortstack{global\\value $P_{t,k-1}$}
& $U'(x_k-)$
& $U'(x_k+)$\\
\hline
5 & 9 & 7 & 0.4148126 & 0.1807358 & 0.512117055 & 0.512695312 & 0.175798 & $-0.587776$ \\
8 & 26 & 14 & 0.0815432 & 0.3842764 & 0.444524450 & 0.444572526 & 0.112327 & $-0.254483$ \\
11 & 52 & 22 & 0.0056812 & 0.4239735 & 0.420045258 & 0.420055694 & 0.082887 & $-0.159456$ \\
14 & 87 & 31 & 0.1582415 & 0.2546399 & 0.407419674 & 0.407423132 & 0.064648 & $-0.117448$ \\
17 & 131 & 40 & 0.2043730 & 0.1988584 & 0.399716018 & 0.399717474 & 0.052380 & $-0.094800$ \\
20 & 184 & 49 & 0.1624330 & 0.2345426 & 0.394525349 & 0.394526064 & 0.043394 & $-0.080648$ \\
23 & 246 & 58 & 0.0556943 & 0.3369005 & 0.390790315 & 0.390790705 & 0.036512 & $-0.071025$ \\
26 & 317 & 68 & 0.2732301 & 0.1161276 & 0.387973825 & 0.387974056 & 0.031025 & $-0.064040$ \\
29 & 397 & 77 & 0.0563066 & 0.3305622 & 0.385774102 & 0.385774248 & 0.026565 & $-0.058770$ \\
32 & 486 & 87 & 0.1764739 & 0.2084221 & 0.384008525 & 0.384008621 & 0.022845 & $-0.054639$ \\
35 & 584 & 97 & 0.2616577 & 0.1216362 & 0.382560100 & 0.382560165 & 0.019706 & $-0.051328$ \\
38 & 691 & 107 & 0.3121614 & 0.0698058 & 0.381350405 & 0.381350452 & 0.017013 & $-0.048611$ \\
41 & 807 & 117 & 0.3298970 & 0.0509537 & 0.380324910 & 0.380324944 & 0.014675 & $-0.046339$ \\
44 & 932 & 127 & 0.3175937 & 0.0623042 & 0.379444515 & 0.379444541 & 0.012624 & $-0.044412$ \\
47 & 1066 & 137 & 0.2782951 & 0.1007803 & 0.378680459 & 0.378680478 & 0.010810 & $-0.042758$ \\
50 & 1209 & 147 & 0.2150595 & 0.1632988 & 0.378011107 & 0.378011122 & 0.009195 & $-0.041321$ \\
53 & 1361 & 157 & 0.1307928 & 0.2469346 & 0.377419880 & 0.377419892 & 0.007746 & $-0.040063$ \\
56 & 1522 & 167 & 0.0281667 & 0.3490014 & 0.376893854 & 0.376893863 & 0.006440 & $-0.038952$ \\
59 & 1692 & 178 & 0.2775564 & 0.0991126 & 0.376422805 & 0.376422812 & 0.005253 & $-0.037961$ \\
62 & 1871 & 188 & 0.1329862 & 0.2432345 & 0.375998543 & 0.375998549 & 0.004172 & $-0.037074$ \\
65 & 2059 & 199 & 0.3502831 & 0.0255327 & 0.375614426 & 0.375614431 & 0.003183 & $-0.036276$ \\
68 & 2256 & 209 & 0.1694067 & 0.2060418 & 0.375265014 & 0.375265018 & 0.002273 & $-0.035551$ \\
71 & 2462 & 220 & 0.3539798 & 0.0211339 & 0.374945807 & 0.374945810 & 0.001435 & $-0.034893$ \\
74 & 2677 & 230 & 0.1427880 & 0.2320192 & 0.374653049 & 0.374653052 & 0.000657 & $-0.034290$ \\
\hline
\end{tabular}
}
\caption{Coefficients used for showing $I_{K_t^-}$ has a non-global local maximum.
The local maximum is at $x_k=1-1/k$.
The exact rational coefficients are defined
by \eqref{eq:exact-mu-system}; only decimal approximations are displayed.
}
\label{tab:homogeneous-certificates}
\end{table}

{\bf Remark:}
The restriction $j\le23$ comes from the derivative condition rather
than from the verification of the graphon inequality. For the next
value, $t=77$, the corresponding candidate is $k=2901$, and the same
search produces a certificate using $K_{241}$ and $K_{242}$. Its two
coefficients are positive, but the left derivative of the resulting
upper envelope is negative.
Thus our method does not imply a local maximum at $2900/2901$.
We do not know whether $I_{K_{77}^-}$ has a non-global
local maximum.

\section{Acknowledgment}
The authors thank Dhruv Mubayi for introducing them to the problem during a workshop organized by Andrew Suk.
The authors thank 
Felix Christian Clemen, 
Misha Tyomkin,
Florian Pfender,
Michael Wigal,
and Bowen Li
for useful discussions on this topic. 

ChatGPT 5.5 was used by 
Luo for assistance with language editing and
with the preparation and debugging of some SageMath verification code. All
computations  were checked by 
at least  two of the 
authors. The
mathematical statements, and the final text were checked by 
all the 
authors.

\bibliographystyle{plainurl}
\bibliography{references}

\section*{Appendix A} 

The first proof of Theorem~\ref{lem:k} where $k$ is restricted to be an integer we obtained using flag algebras. 
We include it because it is interesting to see a parametric proof in flag algebras and one could prove stability of the extremal construction.

\begin{proof}[Proof of Theorem~\ref{lem:k} for integer $k$]
The upper bound comes from the following flag algebras argument.
We want to find a linear combination of flags $X$ such that $X \geq 0$ and 

\begin{align}
k^4 \cdot \Fe50{2 2 2 1 1 2 2 2 2 2} 
\leq
k^4 \cdot \Fe50{2 2 2 1 1 2 2 2 2 2} 
+X
\leq 15(k-1)(k-2).
\label{eqX}
\end{align}
In the case $k \geq 5$, we use
\begin{align*}
X &= 
\left(
x_1 \Fuuu222
+
x_2 \Fuuu221
+
x_3 \Fuuu112 
\right)
\Bigg(  1 \vc{} - k \Fuu1 \Bigg)
+\beta
\Bigg\llbracket
\Bigg( 
\Flllu112122
 - (k-1) 
 \Flllu111111
 \Bigg)^{2} 
\Bigg\rrbracket
+
\\
&~+
\Bigg\llbracket
\left( \Flllu222122, \Flllu222111, \Flllu221122  \right)
\begin{pmatrix}
  1      &   0  \\
  0      &   1  \\
  2-k    &  -1  
\end{pmatrix}
\begin{pmatrix}
  a      &   c  \\
  c      &   b  \\
\end{pmatrix}
\begin{pmatrix}
  1  &   0  & 2-k \\
  0  &   1  &  -1 \\
\end{pmatrix}
\left( \Flllu222122, \Flllu222111, \Flllu221122  \right)^T
\Bigg\rrbracket,
\end{align*}

where
\begin{align*}
a &= 15(2k^3 - 2k^2 + k - 2)/(k - 1)  &
x_1&= 15(k-1)(k-2) \\ 
b &= 15(k^4 - 3k^3 - k + 2) &
x_2 &= 5(k^4 - 3k^3 - 5k^2 + 12k - 4)/(k - 1) \\
c &= 15/2(k^4 + k^3 - 3k^2 + 4k - 4)/(k - 1) &
x_3 &= 5(2k^5 - 7k^4 - 9k^3 + 25k^2 - 21k + 14)/(k - 1)^2 \\
\beta &= 15(k - 2)/(k - 1).
\end{align*}

All coefficients (except $c$) need to be non-negative and $ab-c^2 \geq 0$ in order to create a positive semidefinite matrix. 
Roots of all the polynomials in the numerators are less than 4.
The semidefinite condition is equivalent to 
\[63 \, k^{6} - 378 \, k^{5} + 585 \, k^{4} - 522 \, k^{3} + 603 \, k^{2} - 324 \, k - 36 \geq 0.\] 
Since this polynomial has largest root $\approx4.113060$, we need $k \geq 5$.
To get rid of the fractions in coefficients of the following calculation, we multiply the equation by $4(k-1)^2$ and carry the calculation to prove \eqref{eqX} as
\begin{align}
4(k-1)^2\left(15(k-1)(k-2) -  k^4 \Fe50{2 2 2 1 1 2 2 2 2 2} 
-X \right) =  \sum_{F \in \mathcal{F}_5} c_{k,F} F \geq 0,     \label{eq:k5plus}
\end{align}
where $c_{k,F}$ are polynomials in $k$ that are non-negative for all $k \geq 4$.

The case $k=3$ is implied by the following upper bound $10/27$ on the global maximum of $I_{K_{2,2,1}}$. We  scale by 27 to have smaller fractions and obtain 
\begin{align*}
27 \, \Fe50{2 2 2 1 1 2 2 2 2 2} \leq &~   27 \, \Fe50{2 2 2 1 1 2 2 2 2 2}   
    + \frac{5}{18} \cdot  \Bigg(  6 \Fuuu111 - 2 \Fuuu122 + 3 \Fuuu222 \Bigg)  ^{2} 
\\
&+ 60\cdot  \left\llbracket 
    F  { 
\text{\tiny{
    $\begin{pmatrix}       4/3 &        2/5 &  -629/3000  & -441/1000   & -512/375       &  2/3   \\
          2/5 &       3/25 & -629/10000 & -1323/10000  &  -256/625      &   1/5 \\
  -629/3000 & -629/10000 &     53/100  &  -121/500  &   -47/300  & -629/6000 \\
  -441/1000 &-1323/10000 &   -121/500  &    59/100  &    93/500  & -441/2000 \\
   -512/375 &   -256/625 &    -47/300  &    93/500  &   334/125  &  -256/375 \\
        2/3 &        1/5 &  -629/6000  & -441/2000  &  -256/375  &       1/3 
 \end{pmatrix}$}}}  F^T \right\rrbracket  
 \leq 10,
\end{align*} 
 
\noindent where the $6 \times 6$ matrix is positive semidefinite and 
\[
F = \begin{pmatrix}  \Flluu211221,  &   \Flluu222222,  &   \Flluu212222,  &   \Flluu222122,  &   \Flluu222221,  &   \Flluu212212 \end{pmatrix}.
\]

The case $k=4$ is implied by the following upper bound $45/128$ on $I_{K_{2,2,1}}(3/4)$. We scale by 128 to have smaller numbers in the denominators of the fractions and obtain 

\begin{align*}
128 \Fe50{1 2 2 2 2 2 2 1 2 2} \leq&~     
128 \Fe50{1 2 2 2 2 2 2 1 2 2} +  \Bigg( 190 \Fe30{1 2 2}  - 60 \Fe30{1 1 1} \Bigg) \times  \Bigg(  \Fe20{2} - \frac{3}{4} \Bigg) \\
&~
+ \frac{15}{2}\cdot  \left\llbracket F  \begin{pmatrix} 91  &  12  &  -115  \\ 
 12  &  41  &  -94  \\ 
 -115  &  -94  &  303 \end{pmatrix}  F^T\right\rrbracket \leq 45,
\end{align*}
where the $3 \times 3$ matrix is positive semidefinite and 
\[
F = \begin{pmatrix}  \Fe43{1 2 1 2 1 2},  &  \Fe43{1 2 2 2 2 2},  &  \Fe43{1 2 2 2 2 1} \end{pmatrix}.
\]
\end{proof}

\section*{Appendix B} 

The expansion of $128 \sum_{F \in \mathcal{F}_6} c_F F$ from Lemma~\ref{lem:0.74}, where numbers $c_F$ were rounded up to the 3 most significant digits after the decimal point.

$
-98.737 \Fe60{1 1 1 1 1 1 1 1 2 1 1 2 1 2 2}
-66.502 \Fe60{1 1 1 1 2 1 1 1 2 1 1 2 2 1 2}
-56.616 \Fe60{1 1 1 2 2 1 1 2 2 2 1 1 2 2 1}
-55.942 \Fe60{1 1 1 1 2 1 1 2 2 2 1 2 1 2 2}
-52.717 \Fe60{1 1 1 2 2 1 2 2 2 2 2 2 1 1 1}
-51.458 \Fe60{1 2 2 2 2 2 2 2 2 2 2 2 2 2 2}\\
-50.672 \Fe60{1 1 1 2 2 1 2 2 2 2 2 2 1 1 2}
-44.419 \Fe60{1 1 1 2 2 1 1 2 2 2 1 2 1 2 1}
-42.457 \Fe60{1 1 1 1 2 1 1 2 1 2 1 2 1 2 2}
-42.249 \Fe60{1 1 1 2 2 1 2 1 2 2 2 2 1 1 1}
-41.927 \Fe60{1 1 1 2 2 1 1 2 2 2 1 2 1 2 2}
-40.182 \Fe60{1 1 1 1 1 1 1 1 2 1 1 2 2 2 2}\\
-35.138 \Fe60{1 1 1 1 2 1 1 1 2 1 2 1 2 1 2}
-34.914 \Fe60{1 1 1 1 2 1 1 2 1 2 1 2 2 1 2}
-29.75 \Fe60{1 1 1 2 2 1 1 2 2 2 1 1 2 2 2}
-29.531 \Fe60{1 1 1 1 2 1 1 1 2 1 2 2 2 2 1}
-23.452 \Fe60{1 1 1 2 2 1 1 2 2 2 1 1 1 2 1}
-23.252 \Fe60{1 1 2 2 2 2 2 2 2 2 2 2 2 2 2}\\
-22.372 \Fe60{1 1 1 1 1 1 1 2 2 1 2 2 2 2 1}
-21.896 \Fe60{1 1 1 2 2 1 2 1 2 2 2 2 1 1 2}
-20.241 \Fe60{1 1 1 1 1 1 1 2 2 1 2 2 2 2 2}
-19.231 \Fe60{1 1 1 1 2 1 1 1 2 1 1 2 2 2 2}
-16.591 \Fe60{1 1 1 1 2 1 1 2 2 2 1 2 2 2 2}
-15.952 \Fe60{1 1 1 1 2 1 1 2 1 2 1 1 1 2 2}\\
-15.243 \Fe60{1 1 1 2 2 1 2 1 2 2 2 2 1 2 2}
-14.011 \Fe60{1 1 1 1 1 1 1 1 2 1 1 2 2 1 2}
-13.43 \Fe60{1 1 1 1 2 1 1 1 2 1 2 2 2 2 2}
-13.333 \Fe60{1 1 1 1 2 1 1 2 1 2 1 2 2 2 2}
-12.757 \Fe60{1 1 1 2 2 1 1 2 2 2 1 1 1 2 2}
-12.498 \Fe60{1 1 1 1 1 1 1 1 2 1 2 2 2 2 1}\\
-11.519 \Fe60{1 1 1 1 2 1 2 2 1 2 2 2 1 2 2}
-10.583 \Fe60{1 1 1 2 2 2 2 1 1 2 1 1 1 1 2}
-10.543 \Fe60{1 1 1 1 1 1 1 1 1 1 1 2 1 2 2}
-8.661 \Fe60{1 1 1 2 2 2 2 1 1 2 1 1 1 2 2}
-8.561 \Fe60{1 1 1 1 1 1 1 1 2 1 2 2 2 2 2}
-7.628 \Fe60{1 1 1 1 2 1 1 1 2 1 2 1 2 2 2}\\
-5.305 \Fe60{1 1 2 2 2 1 2 2 2 2 2 2 1 1 2}
-4.925 \Fe60{1 1 1 1 2 1 1 2 2 2 1 2 2 1 2}
-4.025 \Fe60{1 1 1 2 2 2 2 1 1 2 1 2 1 2 2}
-3.982 \Fe60{1 1 2 2 2 2 1 2 2 2 1 2 2 1 2}
-1.883 \Fe60{1 1 1 2 2 2 2 1 2 2 2 1 2 2 1}
-0.965 \Fe60{1 1 1 2 2 1 1 2 2 2 1 2 2 2 2}\\
-0.269 \Fe60{1 1 1 2 2 2 2 2 2 2 2 2 2 2 2}
+0.08 \Fe60{1 1 1 1 2 1 2 2 2 2 2 2 1 2 2}
+0.46 \Fe60{1 1 1 1 1 1 1 2 2 2 1 2 1 2 2}
+0.797 \Fe60{1 1 1 2 2 2 2 1 1 2 1 2 2 2 2}
+1.029 \Fe60{1 1 1 1 2 1 2 2 1 2 2 2 1 1 2}
+1.059 \Fe60{1 1 1 2 2 2 2 1 1 2 1 2 2 1 2}\\
+2.199 \Fe60{1 1 1 1 2 1 1 2 2 2 2 1 2 1 2}
+2.991 \Fe60{1 1 1 1 2 1 1 2 1 2 1 1 2 2 2}
+3.381 \Fe60{1 1 1 2 2 2 2 1 2 2 1 2 1 2 2}
+3.414 \Fe60{1 1 1 2 2 1 2 1 2 2 2 1 1 1 2}
+3.842 \Fe60{1 1 1 2 2 1 1 2 2 2 1 2 2 2 1}
+4.335 \Fe60{1 1 1 2 2 2 2 1 2 2 1 2 2 2 2}\\
+4.598 \Fe60{1 1 1 2 2 1 1 2 2 2 1 1 1 1 1}
+5.284 \Fe60{1 1 1 2 2 1 2 2 2 2 2 2 1 2 1}
+6.122 \Fe60{1 1 1 1 2 1 1 2 2 2 2 1 2 1 1}
+6.735 \Fe60{1 1 1 2 2 2 2 1 2 2 1 2 2 1 1}
+6.767 \Fe60{1 1 1 2 2 1 2 1 2 2 2 2 2 1 2}
+8.629 \Fe60{1 1 1 1 1 1 1 1 1 1 2 2 2 2 1}\\
+8.863 \Fe60{1 1 1 1 2 1 2 2 1 2 2 2 2 1 1}
+8.883 \Fe60{1 1 1 1 2 1 1 2 1 2 1 2 2 2 1}
+9.158 \Fe60{1 1 1 2 2 1 2 1 2 2 2 1 1 2 2}
+9.344 \Fe60{1 1 1 2 2 2 2 1 2 2 2 2 2 2 1}
+9.693 \Fe60{1 1 1 2 2 1 2 1 2 2 2 1 2 2 2}
+10.163 \Fe60{1 1 1 1 2 1 2 2 1 2 2 2 2 1 2}\\
+10.405 \Fe60{1 1 1 2 2 2 2 1 2 2 1 2 2 1 2}
+10.602 \Fe60{1 1 2 2 2 2 1 2 2 2 1 2 1 2 2}
+10.686 \Fe60{1 1 1 1 2 1 1 1 2 1 1 2 2 1 1}
+10.959 \Fe60{1 1 2 2 2 2 1 2 2 2 2 2 2 2 1}
+11.007 \Fe60{1 1 1 1 2 1 1 2 1 2 1 2 2 1 1}
+11.017 \Fe60{1 1 1 2 2 2 2 1 1 2 2 2 2 2 2}\\
+11.172 \Fe60{1 1 1 2 2 2 2 1 2 2 2 2 2 2 2}
+12.499 \Fe60{1 1 1 1 2 1 1 1 2 2 2 1 2 1 1}
+12.596 \Fe60{1 1 1 2 2 1 2 2 2 2 2 2 1 2 2}
+12.682 \Fe60{1 1 1 1 1 1 1 1 2 2 2 1 2 1 2}
+13.161 \Fe60{1 1 1 1 1 1 1 2 2 2 1 2 2 2 2}
+13.3 \Fe60{1 1 1 1 1 1 1 1 2 1 2 1 2 2 2}\\
+13.374 \Fe60{1 1 1 1 2 1 1 1 2 2 2 1 2 1 2}
+13.404 \Fe60{1 1 1 1 2 1 1 2 2 1 2 2 2 2 2}
+13.94 \Fe60{1 1 1 1 2 1 1 1 2 1 2 1 2 2 1}
+13.997 \Fe60{1 1 1 1 2 1 1 2 1 2 1 2 1 2 1}
+14.455 \Fe60{1 1 2 2 2 2 1 2 2 2 2 2 2 2 2}
+14.818 \Fe60{1 1 2 2 2 2 2 2 2 2 2 2 1 1 2}\\
+15.091 \Fe60{1 1 1 1 1 1 1 1 2 1 2 1 2 2 1}
+15.202 \Fe60{1 1 1 1 1 1 1 1 2 2 2 1 2 2 2}
+15.558 \Fe60{1 1 1 1 2 2 2 2 1 2 2 1 2 2 2}
+15.613 \Fe60{1 1 1 2 2 2 2 1 2 2 1 2 2 2 1}
+15.758 \Fe60{1 1 1 2 2 1 2 1 2 2 2 2 2 1 1}
+15.904 \Fe60{1 1 1 1 2 1 1 2 2 2 1 2 2 1 1}\\
+16.656 \Fe60{1 1 1 1 1 1 1 2 2 2 1 2 2 1 2}
+16.94 \Fe60{1 1 1 1 2 1 1 2 1 2 2 2 2 2 1}
+17.013 \Fe60{1 1 2 2 2 2 1 2 2 1 2 2 2 2 2}
+17.093 \Fe60{1 1 1 2 2 1 1 2 2 2 1 1 1 1 2}
+17.669 \Fe60{1 1 1 1 2 1 1 1 2 1 2 1 2 1 1}
+17.926 \Fe60{1 1 1 1 2 1 1 2 2 1 2 2 2 2 1}\\
+17.932 \Fe60{1 1 1 1 2 1 1 2 2 2 2 1 2 2 2}
+17.957 \Fe60{1 1 1 2 2 1 1 2 2 2 1 2 2 1 2}
+18.062 \Fe60{1 1 1 2 2 1 1 2 2 2 1 2 2 1 1}
+18.345 \Fe60{1 1 1 1 1 1 1 2 2 2 1 2 2 1 1}
+19.753 \Fe60{1 1 1 1 2 1 1 2 2 2 2 1 2 2 1}
+19.949 \Fe60{1 1 1 1 2 2 2 2 1 2 2 2 2 2 2}\\
+20.701 \Fe60{1 1 1 1 1 1 1 1 1 1 1 2 2 1 2}
+20.724 \Fe60{1 1 1 1 2 1 1 2 1 2 1 1 1 1 2}
+20.778 \Fe60{1 1 1 1 2 2 2 2 1 2 2 1 2 1 2}
+21.371 \Fe60{1 1 2 2 2 2 1 2 2 1 2 2 2 2 1}
+21.559 \Fe60{1 1 1 2 2 2 2 1 2 2 2 1 2 2 2}
+22.321 \Fe60{1 1 1 1 1 1 1 1 1 1 1 2 2 2 2}\\
+22.379 \Fe60{1 1 1 1 1 1 1 1 1 1 2 2 2 2 2}
+22.937 \Fe60{1 1 2 2 2 2 1 2 2 2 1 2 2 1 1}
+23.158 \Fe60{1 1 1 2 2 1 2 2 2 2 2 2 2 2 1}
+23.509 \Fe60{1 1 1 1 2 1 2 2 2 2 2 2 2 1 1}
+23.663 \Fe60{1 1 1 1 2 1 1 1 2 2 2 1 2 2 2}
+24.161 \Fe60{1 1 1 1 2 1 2 2 1 2 2 2 2 2 2}\\
+25.148 \Fe60{1 1 1 1 1 1 1 2 2 2 2 2 2 2 1}
+25.422 \Fe60{1 1 1 1 1 1 1 1 2 1 1 2 2 1 1}
+26.423 \Fe60{1 1 1 1 1 1 1 1 2 2 2 1 2 1 1}
+26.605 \Fe60{1 1 2 2 2 2 1 2 2 2 1 2 2 2 2}
+27.925 \Fe60{1 1 1 1 2 1 1 2 1 1 2 2 2 2 1}
+28.355 \Fe60{1 1 1 2 2 1 2 1 2 2 2 2 2 2 2}\\
+28.732 \Fe60{1 1 1 2 2 2 2 2 2 2 2 2 2 2 1}
+28.805 \Fe60{1 1 1 1 2 1 1 2 1 2 2 2 2 2 2}
+30.062 \Fe60{1 1 1 1 2 1 1 2 1 1 2 2 2 2 2}
+30.222 \Fe60{1 1 1 1 1 1 1 2 2 2 2 2 2 2 2}
+31.883 \Fe60{1 1 1 1 1 1 1 1 1 1 1 1 1 2 2}
+33.594 \Fe60{1 1 1 1 2 1 1 2 2 2 2 2 2 2 1}\\
+33.791 \Fe60{1 1 1 1 2 2 2 2 2 2 2 2 2 2 2}
+33.824 \Fe60{1 1 1 1 1 1 1 1 2 2 2 2 2 2 2}
+33.98 \Fe60{1 1 1 1 1 1 2 2 2 2 2 2 2 2 2}
+35.38 \Fe60{1 1 1 1 2 1 1 2 2 2 2 2 2 2 2}
+35.764 \Fe60{1 1 1 1 2 2 2 2 1 2 2 1 2 1 1}
+37.379 \Fe60{1 1 1 1 2 1 2 2 2 2 2 2 2 2 2}\\
+38.746 \Fe60{1 1 1 1 1 1 1 1 1 1 1 2 2 1 1}
+41.425 \Fe60{1 1 1 1 1 1 2 2 2 2 2 2 1 2 2}
+42.666 \Fe60{1 1 1 2 2 1 1 2 2 2 2 2 2 2 1}
+44.026 \Fe60{1 1 1 1 2 1 1 2 1 2 1 1 1 1 1}
+44.603 \Fe60{1 1 1 2 2 1 2 2 2 2 2 2 2 2 2}
+44.898 \Fe60{1 1 1 1 1 1 1 1 1 1 1 1 1 1 1}\\
+44.907 \Fe60{1 1 1 1 1 1 1 1 1 1 1 1 1 1 2}
+44.911 \Fe60{1 1 2 2 2 1 2 2 2 2 2 2 1 1 1}
+44.912 \Fe60{2 2 2 2 2 2 2 2 2 2 2 2 2 2 2}
+44.925 \Fe60{1 1 1 1 1 1 1 1 1 1 1 1 2 2 2}
+44.926 \Fe60{1 2 2 2 2 2 2 2 2 1 2 2 2 2 2}
+44.927 \Fe60{1 1 1 2 2 1 1 2 2 1 2 2 2 2 1}\\
+44.927 \Fe60{1 1 1 1 2 1 1 1 2 2 2 2 2 2 2}
+44.929 \Fe60{1 2 2 2 2 2 2 2 2 1 2 2 2 2 1}
+44.934 \Fe60{1 1 2 2 2 2 2 2 2 2 2 2 1 2 2}
+44.934 \Fe60{1 1 1 1 2 1 2 2 2 2 2 2 2 1 2}
+44.935 \Fe60{1 1 2 2 2 1 2 2 2 2 2 2 2 2 2}
+44.935 \Fe60{1 1 1 1 2 1 1 1 2 1 1 2 1 2 2}\\
+44.94 \Fe60{1 1 1 2 2 1 1 2 2 1 2 2 2 2 2}
+44.94 \Fe60{1 1 1 2 2 1 1 2 2 2 2 2 2 2 2}
+44.94 \Fe60{1 1 1 1 1 2 2 2 2 2 2 2 2 2 2}
+44.943 \Fe60{1 1 1 1 1 1 1 1 1 2 2 2 2 2 2}
+44.947 \Fe60{1 1 2 2 2 1 2 2 2 2 2 2 1 2 2}
$

\section*{Appendix C}

\section{\texorpdfstring{Expansion of \eqref{eq:k5plus}}{Expansion of ckf}}

The expansion of \eqref{eq:k5plus} as 
$ \sum_{F \in \mathcal{F}_5} c_{k,F} F$ is

$
4(k-1)^2\left(15(k-1)(k-2) -  k^4 \Fe50{2 2 2 1 1 2 2 2 2 2} 
-X \right)
=\\
~( 12 \, k^{6} - 54 \, k^{5} + 42 \, k^{4} - 66 \, k^{3} + 210 \, k^{2} - 168 \, k + 24 )\,\, \Fe50{1 1 1 1 1 1 1 1 1 2} \quad[ 3.67637881255240 ] \\  
+( 18 \, k^{6} - 82 \, k^{5} + 48 \, k^{4} + 10 \, k^{3} + 106 \, k^{2} - 100 \, k )\,\, \Fe50{1 1 1 1 1 1 1 1 2 2} \quad[ 3.71057759546278 ] \\  
+( 18 \, k^{6} - 84 \, k^{5} + 60 \, k^{4} + 24 \, k^{3} + 42 \, k^{2} - 60 \, k )\,\, \Fe50{1 1 1 1 1 1 2 1 2 2} \quad[ 3.61153242453219 ] \\  
+( 6 \, k^{6} - 34 \, k^{5} + 62 \, k^{4} - 16 \, k^{3} + 38 \, k^{2} - 80 \, k + 8 )\,\, \Fe50{1 1 1 2 1 1 2 2 2 2} \quad[ 1.19111047336275 ] \\  
+( 14 \, k^{6} - 76 \, k^{5} + 100 \, k^{4} - 10 \, k^{3} + 24 \, k^{2} - 44 \, k - 24 )\,\, \Fe50{1 1 1 2 1 1 2 2 1 2} \quad[ 3.33334307995148 ] \\  
+( 16 \, k^{6} - 82 \, k^{5} + 84 \, k^{4} + 16 \, k^{3} + 32 \, k^{2} - 46 \, k - 36 )\,\, \Fe50{1 1 1 1 1 1 2 2 1 2} \quad[ 3.53889649274770 ] \\  
+( 16 \, k^{6} - 80 \, k^{5} + 72 \, k^{4} + 8 \, k^{3} + 72 \, k^{2} - 56 \, k - 48 )\,\, \Fe50{1 1 1 1 1 1 2 2 1 1} \quad[ 3.66913479985175 ] \\  
+( 20 \, k^{6} - 100 \, k^{5} + 76 \, k^{4} + 80 \, k^{3} - 54 \, k^{2} + 46 \, k - 84 )\,\, \Fe50{1 1 1 2 1 1 2 2 1 1} \quad[ 3.74585343172358 ] \\  
+( 16 \, k^{6} - 76 \, k^{5} + 48 \, k^{4} + 100 \, k^{3} - 64 \, k^{2} + 8 \, k - 48 )\,\, \Fe50{1 1 1 1 1 1 2 2 2 2} \quad[ 3.46001458402459 ] \\  
+( 24 \, k^{6} - 102 \, k^{5} + 192 \, k^{3} - 108 \, k^{2} + 54 \, k - 60 )\,\, \Fe50{1 1 1 1 1 1 1 2 2 2} \quad[ 3.76112866206846 ] \\  
+( 18 \, k^{6} - 94 \, k^{5} + 92 \, k^{4} + 54 \, k^{3} - 62 \, k^{2} + 48 \, k - 72 )\,\, \Fe50{1 1 1 2 1 2 1 2 2 1} \quad[ 3.66599104818950 ] \\  
+( 12 \, k^{6} - 62 \, k^{5} + 68 \, k^{4} + 40 \, k^{3} - 40 \, k^{2} + 2 \, k - 36 )\,\, \Fe50{1 1 1 2 1 2 1 2 2 2} \quad[ 3.13907463335862 ] \\  
+( 13 \, k^{6} - 74 \, k^{5} + 130 \, k^{4} - 63 \, k^{3} + 34 \, k^{2} - 40 \, k - 16 )\,\, \Fe50{1 1 1 2 1 2 2 2 2 1} \quad[ 1.19634329488937 ] \\  
+( 4 \, k^{6} - 23 \, k^{5} + 48 \, k^{4} - 4 \, k^{3} + 7 \, k^{2} - 44 \, k - 4 )\,\, \Fe50{1 1 1 2 1 2 2 2 2 2} \quad[ 1.19430945100174 ] \\  
+( 8 \, k^{6} - 56 \, k^{5} + 132 \, k^{4} - 76 \, k^{3} + 76 \, k^{2} - 92 \, k - 24 )\,\, \Fe50{1 1 1 1 1 2 2 2 2 1} \quad[ 1.19067676483104 ] \\  
+( 8 \, k^{6} - 42 \, k^{5} + 48 \, k^{4} + 84 \, k^{3} - 76 \, k^{2} + 6 \, k - 60 )\,\, \Fe50{1 1 1 1 1 2 2 2 2 2} \quad[ 1.19091776236183 ] \\  
+( 20 \, k^{6} - 100 \, k^{5} + 80 \, k^{4} + 60 \, k^{3} - 90 \, k^{2} + 90 \, k - 60 )\,\, \Fe50{1 1 2 2 2 1 2 2 1 1} \quad[ 3.80578666654232 ] \\  
+( 144 \, k^{3} - 144 \, k^{2} + 48 \, k - 96 )\,\, \Fe50{1 1 1 1 2 2 2 2 2 2} \quad[ 1.19042719688608 ] \\  
+( 24 \, k^{6} - 120 \, k^{5} + 72 \, k^{4} + 192 \, k^{3} - 216 \, k^{2} + 168 \, k - 144 )\,\, \Fe50{1 1 1 2 2 2 1 2 1 1} \quad[ 3.79093801524022 ] \\  
+( 16 \, k^{6} - 88 \, k^{5} + 100 \, k^{4} + 68 \, k^{3} - 106 \, k^{2} + 70 \, k - 84 )\,\, \Fe50{1 1 1 2 2 2 1 2 1 2} \quad[ 3.50666702108920 ] \\  
+( 16 \, k^{4} - 8 \, k^{3} - 16 \, k )\,\, \Fe50{1 1 2 2 2 2 2 2 2 2} \quad[ 1.19742933693303 ] \\  
+( 4 \, k^{6} - 28 \, k^{5} + 72 \, k^{4} - 32 \, k^{3} + 2 \, k^{2} - 22 \, k - 12 )\,\, \Fe50{1 1 2 2 2 2 1 2 2 2} \quad[ 1.19318246797484 ] \\  
+( 8 \, k^{6} - 56 \, k^{5} + 120 \, k^{4} - 16 \, k^{3} - 32 \, k^{2} - 8 \, k - 48 )\,\, \Fe50{1 1 2 2 2 2 1 2 1 2} \quad[ 1.19035619811791 ] \\  
+( 12 \, k^{6} - 68 \, k^{5} + 100 \, k^{4} - 24 \, k^{3} - 16 \, k^{2} + 12 \, k - 24 )\,\, \Fe50{1 1 2 2 2 2 1 2 2 1} \quad[ 3.44702691464773 ] \\  
+( 8 \, k^{6} - 48 \, k^{5} + 80 \, k^{4} + 20 \, k^{3} - 52 \, k^{2} + 16 \, k - 48 )\,\, \Fe50{1 1 1 2 2 2 1 2 2 2} \quad[ 1.19195878318966 ] \\  
+( 12 \, k^{4} + 48 \, k^{3} - 54 \, k^{2} + 6 \, k - 36 )\,\, \Fe50{1 1 1 2 2 2 2 2 2 2} \quad[ 1.19232007415003 ] \\  
\\
+ 0 \left( 
\Fe50{2 2 2 1 1 2 2 2 2 2}
+ \Fe50{2 2 2 2 2 2 2 2 2 2}
+ \Fe50{2 2 2 1 2 2 2 2 2 2}
+ \Fe50{1 1 2 2 2 2 2 2 2 1}
+ \Fe50{1 1 1 1 1 1 1 1 1 1}
+ \Fe50{1 1 1 2 1 1 2 1 2 2}
+ \Fe50{1 1 2 2 1 2 2 2 2 1}
+ \Fe50{1 1 2 2 1 2 2 2 2 2}
\right),
$\\

\noindent where the numbers in the square brackets are the largest roots of the polynomials on respective lines computed by SageMath.
Since the largest root is less than 4, all coefficients are non-negative for $k \geq 4$.
Notice that the coefficient is 0 for subgraphs of the extremal construction. 
This must be the case and it serves as a sanity check.

\end{document}